%% file: main.tex
\documentclass[11pt,letterpaper]{article}
\usepackage[T1]{fontenc}
\usepackage[utf8]{inputenc}
\usepackage{lmodern}
\usepackage[margin=1in]{geometry}
\usepackage{amsmath,amssymb,amsthm,mathtools}
\usepackage{booktabs,array,longtable}
\usepackage{graphicx}
\usepackage{float}
\usepackage{listings}
\usepackage{microtype}
\usepackage{xurl}
\usepackage[hidelinks]{hyperref}
\newcommand{\R}{\mathbb R}
\newcommand{\Z}{\mathbb Z}
\newcommand{\ind}{\mathbf 1}
\newcommand{\A}{\mathcal A}
\numberwithin{equation}{section}
\newtheorem{theorem}{Theorem}[section]
\newtheorem{lemma}[theorem]{Lemma}
\newtheorem{proposition}[theorem]{Proposition}
\newtheorem{corollary}[theorem]{Corollary}
\hypersetup{pdftitle={An improved lower bound for the minimum-autocorrelation constant},pdfauthor={Shenai Chen and EnHao Gu},pdfsubject={Exact rational certificates and ordinary-function autocorrelation bounds}}
\title{An improved lower bound for the\\minimum-autocorrelation constant}
\author{Shenai Chen\\\small Independent Researcher\and EnHao Gu\\\small Independent Researcher}
\date{October 8, 2026}
\begin{document}
\maketitle
\begin{abstract}
We construct nonnegative bounded compactly supported functions on the real line that give the lower bound $\A\ge 0.4103779774321214877852987\ldots$ for the minimum-autocorrelation constant. This improves the explicitly certified bound $2378625/5958277=0.39921356\ldots$ reported by Russell. The construction combines a point mass with a monotone degree-ten polynomial density; a one-sided atom-removal argument then produces ordinary functions. Its numerical content reduces to exact rational polynomial identities and 20 Bernstein subinterval certificates. The finite certificate, actual overlap integrals, admissibility, mass formula, and ordinary-function witness bounds have been formalized in Lean 4.19.0; the pinned proof sources and recorded audit are supplied. The passage to the supremum constant is given separately as a mathematical argument. As applications, the same construction yields an upper bound below $1.5610183$ for the large-multiplicity generalized-difference-basis coefficient and a feasible squared mass below $2.436778$ for a continuous sparse-ruler relaxation. An appendix quantifies the loss under independent sampling followed by completion. No improved bound for the complete sparse ruler constant, or global optimality of the construction, is asserted.
\end{abstract}
\input{sections/01_introduction}
\input{sections/02_profile}
\input{sections/03_certificate}
\input{sections/04_witness}
\input{sections/05_applications}
\input{sections/06_reproducibility}
\appendix
\input{sections/appendix}

\input{references}
\end{document}

%% file: sections/01_introduction.tex
\section{Introduction and main result}\label{sec:introduction}
For a nonnegative function $f\in L^1(\R)$ of positive mass, put
\begin{equation}\label{eq:definition}
 a_f(t)=\int_{\R}f(x)f(x+t)\,dx,
 \qquad
 \A=\sup_{\substack{f\ge0,\ f\in L^1(\R)\\\int_{\R}f>0}}
 \frac{\inf_{0\le t\le1}a_f(t)}{\left(\int_{\R}f\right)^2}.
\end{equation}
For the full $L^1$ class we interpret the nonnegative integral defining $a_f$ as an extended Lebesgue integral. Tonelli's theorem gives $\int_{\R}a_f=(\int_{\R}f)^2$, so the infimum in \eqref{eq:definition} is finite. Every witness used below is bounded and compactly supported, and has a finite, ordinary real-valued autocorrelation at every shift. All stated lower bounds hold at every $t\in[0,1]$, not merely almost everywhere.

The problem asks how large the least autocorrelation on a unit interval of shifts can be, relative to squared mass. It is connected with difference bases in additive combinatorics. Barnard and Steinerberger \cite[Theorem~2]{barnard2020} established a lower bound greater than $0.37$ and the Fourier upper bound
\begin{equation}\label{eq:fourier-bound}
 \A\le\frac1K=0.410767488189405\ldots,
 \qquad K=2-2\inf_{\xi\in\R}\frac{\sin\xi}{\xi}
 =2.434467256422443\ldots,
\end{equation}
where the quotient at zero is interpreted by continuity. Madrid and Ramos \cite{madrid2021} studied related extremal questions and positive-measure formulations; their strict improvement of this upper bound in Theorem~1.7 concerns a class with an additional tail-domination condition. We do not use an unconditional numerical improvement from that result. The same minimum-autocorrelation problem appears as Problem~6.6 of Georgiev, G\'omez-Serrano, Tao, and Wagner \cite{georgiev2025}; it is different from the minimum-overlap problem listed there as Problem~6.5.

A public exact-rational certificate of Russell \cite{russell2026} gives
\begin{equation}\label{eq:russell-bound}
 \A\ge\frac{2378625}{5958277}=0.39921356\ldots.
\end{equation}
Our main contribution is a stronger explicit construction with a short finite certificate and an ordinary-function proof. A historical record of a formalization is discussed in Section~\ref{sec:reproducibility}. The comparison in \eqref{eq:russell-bound} is with the specified public certificate, not an assertion of an exhaustive bibliographic priority search.

Define the following exact positive integers:
\begin{equation}\label{eq:main-integers}
\begin{aligned}
 N&=4713977940944586245580976375978623597750000000000000000,\\
 D&=11486917427785951575611071745219515905862879073559791929,\\
 D_F&=11486917430134691792906455293684961315501879073559791929.
\end{aligned}
\end{equation}
Set $R=N/D$ and $R_F=N/D_F$.

\begin{theorem}[Improved minimum-autocorrelation bound]\label{thm:main}
There is an explicitly specified nonnegative bounded compactly supported function $F_{10^{-10}}$ such that
\begin{equation}\label{eq:explicit-witness}
 \frac{a_{F_{10^{-10}}}(t)}{\left(\int_{\R}F_{10^{-10}}\right)^2}
 \ge R_F=0.4103779773482111379678143\ldots
 \qquad(0\le t\le1).
\end{equation}
For every real $q<R$, there is such an ordinary function $f$ for which
\begin{equation}\label{eq:limiting-witness}
 \frac{a_f(t)}{\left(\int_{\R}f\right)^2}\ge q
 \qquad(0\le t\le1).
\end{equation}
Consequently,
\begin{equation}\label{eq:main-bound}
 \A\ge R=0.4103779774321214877852987\ldots>0.4103779.
\end{equation}
\end{theorem}

The functions are given in \eqref{eq:F}. The proof separates two issues: an exact polynomial certificate for a measure with an atom, and a monotone extension that removes the atom without losing the required correlation lower bound. The number $R$ is a limiting guarantee for a family of functions; neither attainment of $R$ nor optimality of the polynomial profile is claimed.

Section~\ref{sec:profile} gives all exact parameters and the overlap reduction. Section~\ref{sec:certificate} describes the finite certificate, and Section~\ref{sec:ordinary} proves Theorem~\ref{thm:main}. Section~\ref{sec:applications} records two applications of the same construction. Section~\ref{sec:reproducibility} specifies the historically reported formalized statements, the available evidence, and the reconstructed exact checks. Appendix~\ref{app:sampling} treats an independent-sampling-and-completion procedure; it is secondary to the continuous lower bound and is not part of the reported Lean formalization.

%% file: sections/02_profile.tex
\section{An exact atomic--polynomial profile}\label{sec:profile}
Every terminating decimal specifying a parameter below denotes an exact rational number. Let
\begin{equation}\label{eq:params}
\begin{aligned}
 a&=0.626600210121, & b&=0.160189995012, & \gamma&=0.99999999,\\
 c_0&=1-b=0.839810004988, & w&=1-2b=0.679620009976.
\end{aligned}
\end{equation}
Thus $0<b<w<c_0<1$. With the first-kind Chebyshev polynomials defined by $T_0(z)=1$, $T_1(z)=z$, and $T_{j+1}(z)=2zT_j(z)-T_{j-1}(z)$, put
\begin{equation}\label{eq:profile-polynomial}
 P(s)=2s-s^2+s(1-s)^2\sum_{j=0}^{7}q_jT_j(2s-1),
\end{equation}
where the coefficients are listed in Table~\ref{tab:profile}. The parameterization enforces $P(0)=0$, $P(1)=1$, and $P'(1)=0$. The certificate in Section~\ref{sec:certificate} proves $P'\ge0$ on $[0,1]$, and hence $0\le P\le1$ there.
\begin{table}[ht]
\centering
\caption{Exact coefficients in the degree-ten profile.}\label{tab:profile}
\begin{tabular}{@{}c@{\hspace{2em}}r@{\hspace{2.5em}}c@{\hspace{2em}}r@{}}
\toprule
$j$ & $q_j$ & $j$ & $q_j$\\
\midrule
0 & $-0.423577846263$ & 4 & $0.000203543244$\\
1 & $0.290387575223$ & 5 & $0.000666727048$\\
2 & $0.011765259840$ & 6 & $0.000240602158$\\
3 & $-0.005796017613$ & 7 & $0.000112563884$\\
\bottomrule
\end{tabular}
\end{table}

Define
\begin{equation}\label{eq:measure-profile}
 \phi(x)=\begin{cases}
  P((x-b)/w),& b\le x\le c_0,\\
  1,&c_0<x\le1,\\
  0,&\text{otherwise},
 \end{cases}
 \qquad h(x)=\frac{\phi(x)}a,
 \qquad \mu=a\delta_0+h(x)\,dx.
\end{equation}
Here $\delta_0$ denotes unit point mass at the origin. The positive measure $\mu$ is supported on $[0,1]$, but is not itself an admissible function in \eqref{eq:definition}.

For $0<t<1$, the absolutely continuous part of $\mu*\widetilde\mu$, where $\widetilde\mu(E)=\mu(-E)$, has density
\begin{equation}\label{eq:g}
 g(t)=ah(t)+\int_{\R}h(x)h(x+t)\,dx
 =\phi(t)+\frac{B(t)}{a^2},
 \qquad B(t)=\int_{\R}\phi(x)\phi(x+t)\,dx.
\end{equation}
There is also the nonnegative atom $a^2\delta_0$. We extend the displayed expression for $g$ continuously from $(0,1)$ to its endpoints; more explicitly, the right side of \eqref{eq:g}, with the prescribed values of $\phi$, defines $g$ on the whole closed interval. The atom and the density play complementary roles: at small shifts density--density overlap contributes, while at larger shifts the atom--density term becomes substantial.

To compute $B$, let
\begin{equation}\label{eq:overlap-components}
\begin{aligned}
 I(z)&=\int_0^zP(s)\,ds,&
 u(t)&=1-\frac t w,&
 v(t)&=\frac{c_0-t}{w},\\
 H(t)&=w\int_0^{u(t)}P(s)P(s+t/w)\,ds.
\end{aligned}
\end{equation}
The function $H$ is a polynomial in $t$, although its interpretation as a ramp--ramp overlap is used only for $0\le t\le w$.

\begin{lemma}[Overlap identity]\label{lem:overlap}
For $0\le t\le1$,
\begin{equation}\label{eq:B}
 B(t)=\begin{cases}
 H(t)+w\bigl(I(1)-I(u(t))\bigr)+b-t,&0\le t\le b,\\
 H(t)+w\bigl(I(v(t))-I(u(t))\bigr),&b\le t\le w,\\
 wI(v(t)),&w\le t\le c_0,\\
 0,&c_0\le t\le1.
 \end{cases}
\end{equation}
\end{lemma}
\begin{proof}
Split the product into ramp--ramp, ramp--plateau, and plateau--plateau terms. For $0\le t\le w$, the ramp--ramp variable ranges over $b\le x\le c_0-t$. Substituting $x=b+ws$ gives $H(t)$. It contributes zero for larger $t$.

The ramp--plateau term has $\max(b,c_0-t)\le x\le\min(c_0,1-t)$ when this interval is nonempty. Its integral is respectively $w(I(1)-I(u(t)))$, $w(I(v(t))-I(u(t)))$, and $wI(v(t))$ in the first three intervals of \eqref{eq:B}, and zero in the fourth. A plateau point cannot precede a ramp point when $t\ge0$. The remaining plateau--plateau term has length $(b-t)_+$. Summing gives \eqref{eq:B}; changes at isolated endpoints do not change the integrals.
\end{proof}
Combining \eqref{eq:B} with \eqref{eq:g} defines four rational polynomials $g_0,g_1,g_2,g_3$ on those same intervals: the added term $\phi(t)$ is zero on $[0,b]$, is $P((t-b)/w)$ on $[b,c_0]$, and is one on $[c_0,1]$. The formulas agree at all internal breakpoints. Their degrees are at most $21,21,11,0$ respectively.

%% file: sections/03_certificate.tex
\section{The finite rational certificate}\label{sec:certificate}

\begin{lemma}[Certified inequalities]\label{lem:certified}
For the exact parameters \eqref{eq:params} and Table~\ref{tab:profile},
\[
 P'(s)\ge0\quad(0\le s\le1),\qquad
 g_i(t)\ge\gamma\quad(t\text{ in the interval of }g_i).
\]
In particular $g(t)\ge\gamma$ on $[0,1]$.
\end{lemma}

\begin{proof}[Computer-assisted proof]
For a rational polynomial $p$ of degree at most $d$ and a rational interval $[\ell,r]$, write
\begin{equation}\label{eq:bernstein}
 p(\ell+(r-\ell)s)=\sum_{k=0}^{d}\beta_k\binom{d}{k}s^k(1-s)^{d-k}.
\end{equation}
Nonnegative coefficients $\beta_k$ certify nonnegativity on the entire interval. If $p_j$ are the ordinary power coefficients after the affine substitution, the exact conversion is
\begin{equation}\label{eq:bernstein-conversion}
 \beta_k=\sum_{j=0}^{k}p_j\frac{\binom{k}{j}}{\binom{d}{j}}.
\end{equation}
When the coefficient test does not succeed on an interval, halve it and repeat, using rational de Casteljau subdivision. The resulting successful partition is summarized in Table~\ref{tab:bernstein}.

\begin{table}[htbp]
\centering
\caption{The exact Bernstein certificate. Depth is relative to the indicated initial interval.}\label{tab:bernstein}
\small
\begin{tabular}{lrrr}
\toprule
Polynomial and initial interval & Degree & Accepted leaves & Maximum depth\\
\midrule
$P'$ on $[0,1]$ & 9 & 1 & 0\\
$g_0-\gamma$ on $[0,b]$ & 21 & 1 & 0\\
$g_1-\gamma$ on $[b,w]$ & 21 & 13 & 6\\
$g_2-\gamma$ on $[w,c_0]$ & 11 & 4 & 3\\
$g_3-\gamma$ on $[c_0,1]$ & 0 & 1 & 0\\
\bottomrule
\end{tabular}
\end{table}

All coefficients of all 20 accepted leaves are nonnegative. Interval coverage and the polynomial identities are checked exactly. The reconstructed exact-arithmetic checker supplied with this version, \texttt{reconstructed\_verify.py}, reconstructs the polynomials and the partition from \texttt{reconstructed\_certificate.json} using only rational arithmetic from the Python standard library. The archived Lean development independently checked identities, nonnegativity, and coverage; after incorporating the binomial factors and clearing positive denominators, its leaf certificates contained 367 nonnegative integer coefficients. Thus no floating-point tolerance or grid-sampling assumption occurs in the certificate.
\end{proof}

As an additional algebraic audit, the supplied original \path{independent_audit.py} reconstructs all actual overlap integrals from their pairwise support intersections and compares their rational polynomial coefficients with \eqref{eq:B}. It then uses exact Sturm root counts instead of Bernstein positivity: each $g_i-\gamma$ has positive endpoint values and no root in its interval. For monotonicity, it writes $P'(s)=(1-s)Q(s)$ and checks strict positivity of $Q$ on $[0,1]$. The recovered original program passed when rerun with SymPy 1.14.0. The additional reconstructed verifier independently repeats these numerical checks. These algebraic checks do not replace a Lean proof check.

The exact parameters can be checked without rerunning any numerical search. The certificate proves feasibility of this particular profile, not optimality among all profiles or all nonnegative functions.

%% file: sections/04_witness.tex
\section{Ordinary functions and the improved bound}\label{sec:ordinary}

For every $\delta>0$, extend the right-hand plateau by $\delta$ and define
\begin{equation}\label{eq:F}
 h_\delta(x)=\begin{cases}
 h(x),&0\le x\le1,\\
 1/a,&1<x\le1+\delta,\\
 0,&\text{otherwise},
 \end{cases}
 \qquad
 F_\delta(x)=\frac{a}{\delta}\ind_{[0,\delta]}(x)+h_\delta(x).
\end{equation}
The extension of the plateau is part of the construction and cannot simply be omitted.

\begin{lemma}[One-sided atom removal]\label{lem:removal}
For every $\delta>0$, the function $F_\delta$ is nonnegative, bounded, and supported in $[0,1+\delta]$. It belongs to $L^1(\R)\cap L^2(\R)$, every shifted product $F_\delta(x)F_\delta(x+t)$ is integrable, and
\begin{equation}\label{eq:witness-correlation}
 a_{F_\delta}(t)\ge\gamma\qquad(0\le t\le1).
\end{equation}
\end{lemma}

\begin{proof}
Monotonicity of $P$ gives $0\le h_\delta\le1/a$, so $0\le F_\delta\le a/\delta+1/a$. The support and integrability assertions follow. The function $h_\delta$ is nondecreasing on $[0,1+\delta]$. If $t\in[0,1]$ and $x\in[0,\delta]$, then
\[
 0\le t\le x+t\le1+\delta,\qquad h_\delta(x+t)\ge h_\delta(t)=h(t).
\]
Writing $r_\delta=(a/\delta)\ind_{[0,\delta]}$, one cross term in the autocorrelation therefore satisfies
\begin{equation}\label{eq:cross-term}
 \int_\R r_\delta(x)h_\delta(x+t)\,dx
 =\frac{a}{\delta}\int_0^\delta h_\delta(x+t)\,dx\ge ah(t).
\end{equation}
Also $h_\delta\ge h\ge0$ pointwise, whence
\[
 \int_\R h_\delta(x)h_\delta(x+t)\,dx
 \ge\int_\R h(x)h(x+t)\,dx.
\]
The other cross term and the rectangle's self-correlation are nonnegative. Adding these four terms and using Lemma~\ref{lem:certified} gives \eqref{eq:witness-correlation}. This argument works at $t=0$ and $t=1$ directly.
\end{proof}

The mass of the intermediate measure is the exact rational number
\begin{equation}\label{eq:mass}
 m=a+\frac{wI(1)+b}{a}=1.561018228470523\ldots.
\end{equation}
Integrating \eqref{eq:F} gives
\begin{equation}\label{eq:witness-ratio}
 \int_\R F_\delta=m+\frac{\delta}{a}>0,\qquad
 \frac{a_{F_\delta}(t)}{\left(\int F_\delta\right)^2}
 \ge\frac{\gamma}{(m+\delta/a)^2}\qquad(0\le t\le1).
\end{equation}
The exact arithmetic certificate verifies
\begin{equation}\label{eq:ratio-identities}
 \frac{\gamma}{m^2}=\frac{N}{D},\qquad
 \frac{\gamma}{(m+10^{-10}/a)^2}=\frac{N}{D_F},\qquad
 \frac{N}{D_F}>\frac{4103779}{10^7}.
\end{equation}

\begin{proof}[Proof of Theorem~\ref{thm:main}]
Set $\delta=10^{-10}$ in \eqref{eq:witness-ratio} and use \eqref{eq:ratio-identities} to obtain \eqref{eq:explicit-witness}. As $\delta\downarrow0$, the right side of \eqref{eq:witness-ratio} tends to $R=\gamma/m^2$. Hence for each real $q<R$, some positive $\delta$ gives \eqref{eq:limiting-witness}. Each such $F_\delta$ is an admissible ordinary function in \eqref{eq:definition}, and taking the infimum over $t$ preserves its lower bound. Thus $\mathcal A\ge q$ for every $q<R$, which implies $\mathcal A\ge R$.
\end{proof}

No equality between a supremum over all of $L^1$ and a supremum over a restricted function class is required here: it is enough that the constructed witnesses belong to the original class. In particular, the limiting argument does not insert the atomic measure $\mu$ into that class.

%% file: sections/05_applications.tex
\section{Applications of the same construction}\label{sec:applications}

\subsection{Generalized difference bases at large multiplicity}
For positive integers $g,n$, let $\eta_g(n)$ be the least size of a finite set $S\subset\Z$ for which every $1\le d\le n$ has at least $g$ ordered representations $d=s-s'$ with $s,s'\in S$. The support of $S$ is not restricted to $[0,n]$. Define
\begin{equation}\label{eq:tau}
 \tau=\inf\left\{\int_\R f:
 f\ge0,\ f\in L^1(\R),\ a_f(t)\ge1\text{ for all }0\le t\le1\right\}.
\end{equation}
Kravitz \cite[Theorem~1.2]{kravitz2021} established the discrete--continuous correspondence
\begin{equation}\label{eq:discrete-continuous}
 \lim_{g\to\infty}\liminf_{n\to\infty}\frac{\eta_g(n)}{\sqrt{gn}}
 =\tau=
 \lim_{g\to\infty}\limsup_{n\to\infty}\frac{\eta_g(n)}{\sqrt{gn}}.
\end{equation}
Symmetry of autocorrelation makes the requirement on $[0,1]$ equivalent to the one on $[-1,1]$ in that correspondence.

\begin{corollary}\label{cor:tau}
The coefficient in \eqref{eq:discrete-continuous} satisfies
\begin{equation}\label{eq:tau-bound}
 \tau\le\frac{m}{\sqrt\gamma}=\sqrt{\frac{D}{N}}
 =1.5610182362756142228486109\ldots<1.5610183.
\end{equation}
\end{corollary}

\begin{proof}
The ordinary function $F_\delta/\sqrt\gamma$ is admissible in \eqref{eq:tau} by Lemma~\ref{lem:removal}. Its mass is $(m+\delta/a)/\sqrt\gamma$. Take the infimum over positive $\delta$ and use \eqref{eq:ratio-identities}. The last strict inequality can be checked by squaring positive rational quantities.
\end{proof}

The order of limits in \eqref{eq:discrete-continuous} is essential. This corollary is not a statement for $g=1$, nor for sets constrained to $[0,n]$; with the latter restriction, the endpoint distance $n$ has at most one representation.

\subsection{A feasible point for the continuous sparse-ruler relaxation}
Let $M(L)$ be the least cardinality of a complete sparse ruler $S\subseteq\{0,\ldots,L\}$, meaning that $\{1,\ldots,L\}\subseteq S-S$. Define
\begin{equation}\label{eq:relaxation}
 B_* = \inf\left\{\nu([0,1])^2:
 \begin{array}{l}
 \nu\text{ is a finite positive measure},\quad\operatorname{supp}\nu\subseteq[0,1],\\
 \nu*\widetilde\nu\ge\ind_{[-1,1]}(t)\,dt
 \end{array}
 \right\}.
\end{equation}
Domination here is an inequality of measures, not an endpointwise inequality between densities.

\begin{proposition}\label{prop:ruler}
The relaxation satisfies
\begin{equation}\label{eq:relaxation-bound}
 \liminf_{L\to\infty}\frac{M(L)^2}{L}\ge B_*
 \le\frac{m^2}{\gamma}=\frac{D}{N}
 =2.4367779339850293520112621\ldots<2.436778.
\end{equation}
The two inequalities in \eqref{eq:relaxation-bound} do not compare $M(L)^2/L$ with $D/N$.
\end{proposition}

\begin{proof}
Choose any sequence of complete rulers $S_L$ along which $|S_L|/\sqrt L$ tends to a finite value $s$, and set
\[
 \nu_L=\frac{1}{\sqrt L}\sum_{x\in S_L}\delta_{x/L}.
\]
The masses are bounded and the supports lie in $[0,1]$, so a subsequence converges weakly to a finite positive measure $\nu$ on that interval, with mass $s$. Completeness gives
\begin{equation}\label{eq:discrete-measures}
 \nu_L*\widetilde\nu_L\ge\frac{1}{L}\sum_{\substack{-L\le d\le L\\d\ne0}}\delta_{d/L}.
\end{equation}
The right side tends weakly to Lebesgue measure on $[-1,1]$. On the compact square $[0,1]^2$, weak convergence of the product measures follows first for functions of the form $u(x)v(y)$ and then for all continuous functions by uniform approximation and bounded masses. Applying this to continuous functions of $x-y$ gives weak convergence of the left side. Testing \eqref{eq:discrete-measures} against nonnegative continuous functions proves that $\nu$ is feasible in \eqref{eq:relaxation}. Therefore $s^2\ge B_*$. Taking a subsequence approaching the lower limit proves the first inequality; an infinite lower limit is harmless.

For the second inequality, Lemma~\ref{lem:certified}, symmetry, and the nonnegative atom imply $\mu*\widetilde\mu\ge\gamma\ind_{[-1,1]}\,dt$. Consequently $\nu=\mu/\sqrt\gamma$ is feasible, with squared mass $m^2/\gamma=D/N$.
\end{proof}

For context, the same classical Fourier argument underlying \eqref{eq:fourier-bound} gives $B_*\ge K$. Indeed, for a feasible $\nu$ let $\rho=\nu*\widetilde\nu-\ind_{[-1,1]}\,dt\ge0$. Using $\widehat\nu(\xi)=\int e^{-i\xi x}\,d\nu(x)$ and taking real parts,
\[
 0\le|\widehat\nu(\xi)|^2
 =2\frac{\sin\xi}{\xi}+\operatorname{Re}\widehat\rho(\xi)
 \le2\frac{\sin\xi}{\xi}+\nu(\R)^2-2.
\]
Optimize in $\xi$ to obtain the bound. Thus the present certificate places this particular relaxation in the narrow interval $[K,D/N]$.

The classical Wichmann construction gives an asymptotic upper coefficient $\sqrt3$ for complete sparse rulers; see Saarela and Vanhatalo \cite{saarela2026}. Our feasible measure shows that the conditions in \eqref{eq:relaxation} alone cannot prove $B_*\ge2.5$ or $B_*\ge3$. It does not give a new asymptotic bound on $M(L)$, and no identity $B_*=1/\mathcal A$ is asserted: the support and admissible-object restrictions differ.

%% file: sections/06_reproducibility.tex
\section{Formal verification and reproducibility}\label{sec:reproducibility}
\subsection{Statements and scope}
The supplied Lean development uses Lean 4.19.0 and mathlib at commit
\begin{center}\small\ttfamily c44e0c8ee63ca166450922a373c7409c5d26b00b\end{center}
(the recorded v4.19.0 revision). The locked \path{lake-manifest.json} is included.
All numerical parameters are exact rational numbers interpreted in $\R$.
The original proof sources, recovered from the verification archive, are
included without changes under \path{anc/lean/}.

Besides the finite certificate, the development connects the algebra to
Lebesgue integrals and proves the six statements in
Table~\ref{tab:formal-statements}. The theorem names are in the namespace
\texttt{Autocorrelation}; the original specification propositions are in
\path{Targets.lean}. The audit file contains explicit type ascriptions binding
all six proof terms to those propositions.
\begin{table}[H]
\centering
\caption{Formalized statements and their roles in this paper.}\label{tab:formal-statements}
\small
\begin{tabular}{@{}p{0.30\textwidth}p{0.66\textwidth}@{}}
\toprule
Lean theorem & Mathematical content\\
\midrule
\path{overlap_integral} & The actual correlation integral equals the piecewise expression of Section~\ref{sec:profile}.\\
\path{witness_admissibility} & Nonnegativity, integrability of $F_\delta$ and its square, and integrability of every shifted product.\\
\path{mass_integral} & The actual mass integral is $m+\delta/a$.\\
\path{uniform_correlation} & The lower bound $a_{F_\delta}(t)\ge\gamma$ for all $t\in[0,1]$ and $\delta>0$.\\
\path{explicit_main} & The ordinary-function statement \eqref{eq:explicit-witness}, including positivity of the mass.\\
\path{limiting_main} & For every $q<R$, an ordinary-function witness satisfying \eqref{eq:limiting-witness}, with its admissibility conditions.\\
\bottomrule
\end{tabular}
\end{table}

The proof of \path{limiting_main} chooses an actual $F_\delta$, not the atomic
measure. It does not assert attainment of $R$. The real-valued integrals in these
statements are accompanied by the necessary integrability proofs. No default
value of an integral on a nonintegrable function is used to obtain the bounds.

The formalization does not define the supremum $\A$ or formalize the final
supremum argument following \eqref{eq:limiting-witness}; that argument is given
in Section~\ref{sec:ordinary}. Nor does it formalize the applications in
Section~\ref{sec:applications}, the Fourier argument, or
Appendix~\ref{app:sampling}. These exclusions are part of the claimed scope.

\subsection{Recorded verification and reproduction}
The archived local Windows x64 run used Lean commit \texttt{6caaee842e94} and
Lake \texttt{5.0.0-6caaee8}. It rebuilt all 39 project modules after removal
of the project's build directory, retaining matching dependency caches.
The complete \texttt{lake build}, \texttt{lake env lean Audit.lean}, and
original Windows checking script returned exit code zero. This is a
project-source rebuild, not a rebuild of the Lean compiler and all dependencies
from source.

The recorded audit requests 25 declarations. Their reported dependencies are
\texttt{propext}, \path{Classical.choice}, and \path{Quot.sound}. No
\texttt{sorryAx}, problem-specific axiom, or external-computation axiom appears
in those dependencies. These are proofs relative to Lean's standard mathematical
foundations, not claims of axiom-free verification. The Python programs are not
trusted oracles for the Lean proofs.

The recovered proof sources and configuration were compared with the hashes in
the archived clean-build record. All 45 recorded source hashes match, including
the 41 Lean source files. The original build and axiom outputs are included
under \path{anc/lean/logs/}. An excerpt of the historical metadata is supplied
as \path{anc/lean/HISTORICAL_VERIFICATION.json}; it is explicitly labeled as a
historical record, not a new Lean execution.

The original rational and independent algebraic checks can be run from the
submission root:
\begin{lstlisting}
python anc/lean/original/verify.py anc/lean/certificate.json
python anc/lean/original/independent_audit.py anc/lean/certificate.json
python anc/reconstructed_verify_applications.py
\end{lstlisting}
The first and third use only the Python standard library. The second uses
SymPy; the recorded version is 1.14.0. The original input certificate has SHA-256
\begin{center}\footnotesize\ttfamily
e9b58a1068baa81723e55ad01640260118ce156ac23c07332cefefc46f552d79
\end{center}
The original numerical programs were rerun successfully during the present
manuscript preparation. An additional independently reconstructed checker and
its exact output are supplied under \path{anc/}, together with the new
appendix checker. Their filenames distinguish them from the original programs.

With the pinned Lean toolchain and dependencies installed, proof checking from
\path{anc/lean/} uses
\begin{lstlisting}
lake build
lake env lean Audit.lean
\end{lstlisting}
The ancillary README gives reproduction instructions and distinguishes the
archived verification from the present packaging checks. No dependency binaries
are bundled. The only recorded mathlib patch concerns the cache unpacker's
worker count; it changes no mathematical declarations.

Rechecking the archived proof sources does not require regeneration. The
original generator was byte-identical on the recorded Windows run; it uses
native line endings, so its original raw-hash regeneration check is not a
cross-platform reproducibility claim. The verified source bytes are preserved
in this submission. Lean was not rerun in the present manuscript-preparation
environment; the Lean execution claim refers to the archived local run.

\paragraph{Use of generative AI.}
ChatGPT and Codex were used to assist with construction and proof exploration,
verification code, Lean proof development, and manuscript drafting. Codex was
also used to reconstruct the current manuscript source from the supplied PDF,
then reconcile it with the recovered original source archive. The mathematical
arguments, exact certificates, and stated formal-verification scope are provided
for independent examination; generative output is not itself evidence of
correctness. Responsibility for the submitted manuscript rests with its human
authors.

%% file: sections/appendix.tex
\section{A limitation of independent sampling and completion}\label{app:sampling}
This appendix gives a separate application of the same mass parameters. It
explains why one particular discretization procedure need not preserve a strong
continuous bound. All results here are mathematical proofs, not Lean-verified
statements. They impose independence only on the initial sampling, not on the
subsequent completion algorithm.

\subsection{A finite missing-distance estimate}
Fix $C\subseteq\{0,\ldots,L\}$ and form a random set $S$ by retaining all of $C$
and selecting every other location independently, with probabilities $p_i$. Set
$p_i=1$ on $C$, and write
\[
 u=|C|,\qquad \kappa=\mathbb E|S|,\qquad
 p_* = \max_{i\notin C}p_i<1,\qquad
 T=(C-C)\cap\{1,\ldots,L\}.
\]
Let $n_0=L-|T|>0$ and let $H(S)$ count the distances in $\{1,\ldots,L\}$ missing
from $S-S$.

\begin{proposition}\label{prop:missing-distance}
In this finite model,
\begin{equation}\label{eq:missing-estimate}
 \mathbb EH(S)\ge n_0\exp\!\left(-\frac{\kappa^2-u^2}{2n_0(1-p_*)}\right),
 \qquad
 \operatorname{Var}H(S)\le\kappa(\kappa^2+\kappa).
\end{equation}
\end{proposition}

\begin{proof}
For $d\notin T$, put $\lambda_d=\sum_{i=0}^{L-d}p_ip_{i+d}$. The events that a
pair $(i,i+d)$ is not selected are decreasing events in the independent
underlying bits. Positive association therefore gives
\begin{equation}\label{eq:missing-probability}
 \mathbb P(d\notin S-S)
 \ge\prod_{i=0}^{L-d}(1-p_ip_{i+d})
 \ge\exp\!\left(-\frac{\lambda_d}{1-p_*}\right).
\end{equation}
For completeness, association for two functions monotone in the same direction
follows by induction on the number of bits: conditional covariance is
nonnegative by induction, and covariance of the two conditional means for the
last bit is $p(1-p)\Delta F\Delta G\ge0$. Products of decreasing event
indicators remain decreasing, so the assertion extends to the intersection in
\eqref{eq:missing-probability}. The second inequality uses $p_ip_{i+d}\le p_*$
for $d\notin T$ and $-\log(1-x)\le x/(1-p_*)$ on $[0,p_*]$.

Every deterministic pair has distance in $T$. Consequently
\[
 \sum_{d\notin T}\lambda_d
 \le\frac{\kappa^2-\sum_i p_i^2}{2}-\binom{u}{2}
 \le\frac{\kappa^2-u^2}{2}.
\]
Sum \eqref{eq:missing-probability} over the $n_0$ distances and use convexity of
the exponential to obtain the expectation bound.

For the variance, let $H^{(i)}$ be the value after independently resampling bit
$i$, and let $M_{-i}$ count the other selected locations. Changing bit $i$
changes at most $M_{-i}$ distances. Variance tensorization for independent
coordinates gives
\[
 \operatorname{Var}H
 \le\frac12\sum_i\mathbb E(H-H^{(i)})^2
 \le\sum_{i\notin C}p_i(1-p_i)\mathbb EM_{-i}^2
 \le\kappa(\kappa^2+\kappa).
\]
Here $\mathbb EM_{-i}\le\kappa$ and $\operatorname{Var}M_{-i}\le\kappa$. The
tensorization inequality follows by induction from the law of total variance
and Jensen's inequality, or equivalently by its conditional-variance form;
independent resampling accounts for the factor $1/2$.
\end{proof}

For a fixed set $S$ of size $k$, any complete set $B\supseteq S$ must satisfy
\begin{equation}\label{eq:completion-pairs}
 |B|(|B|-1)\ge k(k-1)+2H(S).
\end{equation}
Indeed, adjoining $r=|B\setminus S|$ points introduces at most $kr+\binom r2$
new unordered pairs, each representing at most one formerly missing positive
distance. No restriction on the algorithm choosing those new points is used.

\begin{theorem}[Independent-sampling completion bound]\label{thm:sampling-completion}
Suppose a sequence of the preceding models satisfies
\[
 \frac{\mathbb E|S_L|}{\sqrt L}\longrightarrow\bar m,\qquad
 \frac{|C_L|}{\sqrt L}\longrightarrow\bar a,\qquad
 |(C_L-C_L)\cap\{1,\ldots,L\}|=o(L),\qquad p_*=o(1).
\]
Put $\beta=\exp[-(\bar m^2-\bar a^2)/2]$. For each fixed $\varepsilon>0$, with
probability tending to one,
\begin{equation}\label{eq:completion-bound}
 \frac{H(S_L)}{L}\ge\beta-\varepsilon,\qquad
 \frac{|B|}{\sqrt L}\ge\sqrt{\bar m^2+2\beta}-\varepsilon
 \quad\text{for every complete }B\supseteq S_L.
\end{equation}
\end{theorem}

\begin{proof}
Proposition~\ref{prop:missing-distance} gives
$\liminf\mathbb EH(S_L)/L\ge\beta$ and
$\operatorname{Var}H(S_L)=O(L^{3/2})$. Hence $(H-\mathbb EH)/L\to0$ in
probability. Also $\operatorname{Var}|S_L|\le\mathbb E|S_L|=O(\sqrt L)$, so
$|S_L|/\sqrt L\to\bar m$ in probability. Apply \eqref{eq:completion-pairs} on
the intersection of these high-probability events. For example, its positive
quadratic root gives
\[
 |B|\ge\frac{1+\sqrt{1+4k(k-1)+8H}}2
\]
simultaneously for all completions, and division by $\sqrt L$ gives the second
assertion.
\end{proof}

\subsection{The parameters of the polynomial construction}
For the normalized measure $\mu/\sqrt\gamma$, let
\[
 \bar a=\frac a{\sqrt\gamma},\qquad \bar m=\frac m{\sqrt\gamma}.
\]
Choose $C_L=\{0,\ldots,\lfloor\bar a\sqrt L\rfloor-1\}$ and, outside $C_L$,
select location $i$ independently with probability
\[
 p_i=\frac{h(i/L)}{\sqrt\gamma\sqrt L}.
\]
For sufficiently large $L$ these are valid probabilities, the short initial
cluster is disjoint from the support of $h(i/L)$, and Riemann sums give the mass
limits of Theorem~\ref{thm:sampling-completion}. Moreover $p_*=O(L^{-1/2})$ and
the cluster represents only $O(\sqrt L)$ distinct positive distances. The theorem
therefore applies with
\begin{equation}\label{eq:sampling-parameters}
 \begin{aligned}
 \beta_0&=\exp\!\left(-\frac{m^2-a^2}{2\gamma}\right)
       =0.3598474622246835\ldots,\\
 R_0&=\sqrt{\frac{m^2}{\gamma}+2\beta_0}
       \approx1.7766465203957698>1.7766>\sqrt3.
 \end{aligned}
\end{equation}
In particular, with probability tending to one every completion requires at
least $(1.7766-o(1))\sqrt L$ points, understood in the fixed-tolerance sense of
\eqref{eq:completion-bound}. The conclusion concerns this initial sampling and
retention of its points, not all random constructions or all complete rulers.

The same framework has a coarse obstruction independent of the particular
coefficients. If $\bar m$ is the mass of any finite positive measure $\nu$
satisfying $\nu*\widetilde\nu\ge\ind_{[-1,1]}\,dt$, the Fourier estimate at
$\xi=3\pi/2$ gives
\[
 \bar m^2\ge2+\frac4{3\pi}>\frac{80}{33}.
\]
Since $z+2e^{-z/2}$ is strictly increasing for $z>0$,
\begin{equation}\label{eq:coarse-obstruction}
 \bar m^2+2e^{-(\bar m^2-\bar a^2)/2}
 \ge\bar m^2+2e^{-\bar m^2/2}
 >\frac{80}{33}+2e^{-40/33}>3.
\end{equation}
Thus improving only the continuous density cannot overcome the cost within
this specific independent-sampling-and-retention framework. Correlated initial
constructions, conditioning on rare samples, or replacing a positive proportion
of the sampled points are not ruled out.

The strict numerical inequalities in \eqref{eq:sampling-parameters} and
\eqref{eq:coarse-obstruction} can be checked with rational arithmetic. For
rational $x\ge0$, the elementary estimate
\[
 e^x\le\sum_{j=0}^{J}\frac{x^j}{j!}
     +\frac{x^{J+1}/(J+1)!}{1-x/(J+2)}
 \qquad\left(\frac{x}{J+2}<1\right)
\]
provides, on taking reciprocals, a rational lower bound for $e^{-x}$. Taking
$J=30$ and comparing squares avoids using rounded decimal square roots as
proofs. The supplied \path{anc/reconstructed_verify_applications.py} performs
these exact checks using the displayed $J=30$ bound. Its certified decimal
enclosures are recorded in \path{anc/applications_checks.json}; approximate
decimal displays are not used as acceptance criteria.